\documentclass[12pt]{amsart}
\usepackage[english]{babel}
\usepackage{amsthm}
\usepackage{inputenc}
\usepackage{amssymb}
\usepackage{color}
\newtheorem{thm}{Theorem}[section]

\newtheorem{prop}[thm]{Proposition}
\newtheorem{defn}[thm]{Definition}

\newtheorem{exam}[thm]{Example}
\begin{document}

\author{Rakhimov I.S., Masutova K.K., Omirov B.A.}

\address{I.S. Rakhimov$^1$ \\ Dept. of Math., FS and Institute for Mathematical Research (INSPEM),
Universiti Putra Malaysia, Malaysia.}

\address{K.K. Masutova$^2$ and B.A. Omirov$^3$ \\ Institute of Mathematics, National University of Uzbekistan,
 Uzbekistan.}
\email {$^1$risamiddin@gmail.com \ \ $^2$kamilyam81@mail.ru \ \ $^3$omirovb@mail.ru}

\title{\bf On derivations of semisimple Leibniz algebras.}

\maketitle

\begin{abstract} In the paper we describe derivations of some classes of Leibniz algebras. It is shown that any derivation of a simple Leibniz algebra can be written as a combination of three derivations. Two of these ingredients are a Lie algebra derivations and the third one can be explicitly described.  Then we show that the similar description can found as well as for a subclass of semisimple Leibniz algebras.
\end{abstract}

\medskip



\section{Introduction}
The interest in study of derivations of algebras goes back to a paper by N. Jacobson \cite{Jacobson1}. There using the nilpotency property of derivations Jacobson  specified one important class of Lie algebras called \emph{characteristically nilpotent}. Since then the characteristically nilpotent class of Lie algebras have been intensively and extensively treated by many authors. Motivated by the progress made in 1971 T.S. Ravisankar \cite{RAVISANKAR} extended the concept of being characteristically nilpotent to other classes of algebras. This approach has been used for the study of Malcev algebras, as well as for associative algebras and their deformation theory \cite{MAKHLOUF}.

In 1939 N.Jacobson proved that the exceptional complex simple Lie algebra $G_2$ of dimension 14 can be represented as the algebra of derivations of the Cayley algebra (see \cite{Jacobson2}). This result increased the interest in analyzing the derivations of Lie algebras. Two years earlier, E. V. Schenkman \cite{SCHENKMAN} had published his derivation tower theorem for centerless Lie algebras, which described in a nice manner the derivation algebras. This theory was not applicable to the nilpotent algebras, as the adjoint representation is not faithful. This fact led to the assumption that the structure of derivations for nilpotent Lie algebras is much more difficult than for classical algebras. In simesimple Lie algebras case the use of derivations sheds a light on their structure.
A Lie algebra derivation defined by right multiplication operator is said to be \emph{inner}. All other derivations are called \emph{outer}. It was proved that any derivation of finite-dimensional semi-simple Lie algebra over a field of characteristic zero is inner (see \cite{Jacobson3}). The outer derivations can be interpreted as elements of the first (co)homology group considering the algebra as a module over itself. Here the derivations are 1-(co)cycles and the inner derivations play the role of 1-(co)boundaries. Note that it was proved that any nilpotent Lie algebra has an outer derivation, i.e., there exists at least one derivation which is not the adjoint operator for a vector of the algebra. We remind also the fact that any Lie algebra over a field of characteristic zero which has non degenerate derivations is nilpotent \cite{Togo1}, \cite{Togo2}.

 Derivations have been used in the geometric study of different classes of algebras also (see the review paper \cite{AncocheaCampoamor} and the papers  \cite{GOZEKHAKIMDJANOV}, \cite{RakhimovAtan} among others).
 The derivations of filiform Leibniz algebras have been treated in \cite{LadraRikhsiboevTurdibaev}, \cite{Omirov},  \cite{RakhimovNashribook}, \cite{RakhimovAL-Nashri1}, \cite{RakhimovAL-Nashri2}.

\section{Preliminaries}

In this section we give some necessary definitions and preliminary results.

\begin{defn} An algebra $(L,[\cdot,\cdot])$ over a field $F$ is called a
Leibniz algebra if it is defined by the
identity
$$[x,[y,z]]=[[x,y],z] - [[x,z],y].$$
\end{defn}

The identity is said to be Leibniz identity. It is a generalization of the Jacobi identity
since under condition of anti-symmetricity of the product ``[$\cdot,\cdot$]'' this identity changes to the Jacobi identity. In fact, the definition above expresses the right Leibniz algebra. The descriptor ``right'' indicates that any right multiplication operator is a derivation of the algebra. In the paper the term ``Leibniz algebra'' will always mean the ``right Leibniz algebra''. The left Leibniz algebra is characterized by the property that any left multiplication operator is a derivation.

Let $L$ be a Leibniz algebra and $I$ be the ideal generated by squares in $L:$ $I=id< [x,x]\  | \ x\in L
>.$ The quotient $L/I$ is said to be the associated Lie algebra of the Leibniz algebra $L.$ The natural
epimorphism $\varphi  : L \rightarrow   L/I$ is homomorphism of Leibniz algebras. The ideal $I$ is the minimal ideal with respect to the
property that the quotient algebra is a Lie algebra.
It is easy to see that the ideal $I$ coincides with the
subspace of $L$ spanned by the squares and $L$ is the left annihilator of $I,$ i.e., $[L,I]=0.$

\begin{defn} A Leibniz algebra $L$ is called a simple if its ideals are only $\{0\}, I, L$ and $[L,L]\neq I.$
\end{defn}
This definition agrees with that of simple Lie algebra, where $I=\{0\}.$

For a given Leibniz algebra $L$ we define the derived sequence as follows:
$$L^1=L,\quad L^{[k+1]}=[L^{[k]},L^{[k]}], \quad k \geq 1.$$
\begin{defn} A Leibniz algebra $L$ is called a solvable, if there exists  $s \in \mathbb N$ such that $L^{[s]}=0.$
\end{defn}

Since the sum of two solvable ideals of a Leibniz algebra $L$ is also solvable, there exists a maximal solvable ideal called solvable radical of $L.$

\begin{defn} A Leibniz algebra $L$ is said to be a semisimple if its solvable radical is equal to $I.$
\end{defn}
Clearly, this definition agrees with the definition of semisimplicity of Lie algebras. Note that a simple Leibniz algebra and a direct sum of simple Leibniz algebras are examples of semisimple Leibniz algebras.

The concept of derivation for a Leibniz algebra $L$ is defined as follows.
\begin{defn}
A linear transformation $d$ of a Leibniz algebra $L$ is said to be a
derivation if for any $x, y\in L$ one has
$$d([x,y])=[d(x),y]+[x, d(y)].$$
\end{defn}
\smallskip

Let $z$ be an element of a Leibniz algebra $L.$ Consider the operator of right multiplication $R_z:L\to L$, defined by $R_z(x)=[x,z].$ Remind that the Leibniz algebra is characterized by the property that any such a right multiplication operator $R_z$ is a derivation of $L.$

The following theorem recently proved by D.Barnes \cite{Barnes}  is an analogue of Levi-Malcev's theorem for Leibniz algebras which will be used in the paper.

\begin{thm} \label{t2}  \label{thmBarnes}
Let $L$ be a finite dimensional Leibniz algebra over a field of
characteristic zero and $R$ be its solvable radical. Then there
exists a semisimple Lie subalgebra $S$ of $L$, such that
$L=S\dot{+}R.$
\end{thm}

We mention a few more results used in the paper. The first of them is well-known Schur's lemma given as follows.
\begin{thm}  Let $G$ be a complex Lie algebra, $U$
and $V$ be irreducible $G$-modules. Then
\begin{itemize}
\item[(i)] Any $G$-module homomorphism $\theta : U \rightarrow V$ is either a trivial or an isomorphism;
\item[(ii)] A linear map $\theta : V \rightarrow V$  is a $G$-module homomorphism if and only
if $\theta = \lambda id_{|_V}$ for some $\lambda \in \mathbb{C}.$
\end{itemize}
\end{thm}
Here is the theorem on structure of modules over semisimple Lie algebras.
\begin{thm} \label{reducible} Let $G$ be a semisimple Lie algebra over a field of characteristic zero. Then every finite dimensional module over $G$ is completely reducible.
\end{thm}

One of the subclasses of Leibniz algebras considered in the paper is specified by the following conditions:
\begin{itemize}
\item[$(a)$] $L/I\cong sl_2^1\oplus sl_2^2;$

\item[$(b)$] $I=I_{1,1}\oplus I_{1,2}$ such that $I_{1,1}, I_{1,2}$ are irreducible $sl_2^1$-modules and  $dimI_{1,1}=dimI_{1,2};$

\item[$(c)$] $I=I_{2,1}\oplus I_{2,2}\oplus...\oplus I_{2,m+1}$ such that $I_{2,k}$ are irreducible $sl_2^2$-modules with $1\leq k\leq m+1$.
\end{itemize}
There is the following classification result on the subclass of Leibniz algebras mentioned above (see \cite{Camacho}).

\begin{thm}  \label{thm2} Any $2(m+4)$-dimensional Leibniz algebra $L$ satisfying the conditions $(a)-(c)$ admits a basis $\{e_1,f_1,h_1,e_2,f_2,h_2,x_0^1,x_1^1,x_2^1,...,x_m^1,x_0^2,x_1^2,x_2^2,...,x_m^2\}$ such that the table of multiplication of $L$ in this basis is represented as follows:
$$L\cong \quad \quad \left\{\begin{array}{lll}
\, [e_i,h_i]=-[h_i,e_i]=2e_i, & \\
\, [e_i,f_i]=-[f_i,e_i]=h_i, &   & \\
\, [h_i,f_i]=-[f_i,h_i]=2f_i, & \\
\, [x_k^i,h_1]=(m-2k)x_k^i, & 0 \leq k \leq m ,\\
\, [x_k^i,f_1]=x_{k+1}^i,  & 0 \leq k \leq m-1, \\
\, [x_k^i,e_1]=-k(m+1-k)x_{k-1}^i, & 1 \leq k \leq m, \\
\, [x_j^1,e_2]=[x_j^2,h_2]=x_j^2,&\\
\, [x_j^1,h_2]=[x_j^2,f_2]=-x_j^1,&
 \end{array}\right.$$
with $1\leq i\leq2$ and $ 0\leq j\leq m.$
\end{thm}

Note that the basis above is constructed by extending the canonical bases $\{e_1,f_1,h_1\}$ and $\{e_2,f_2,h_2\}$ of
$sl_2^1$ and $sl_2^2,$ respectively.

\section{Main Result}

This section contains the main results of the paper on description of derivations of two big classes of Leibniz algebras.

\subsection{Derivations of simple Leibniz algebras}

Let us first consider the following example.

\begin{exam} \label{simple} Let $L$ be a complex $(m+4)$-dimensional $(m\geq 2)$ Leibniz algebra with the basis $\{e,f,h,x_0,x_1,\dots,x_m\}$ such that non zero products of the basis vectors in $L$ are represented as follows:
$$\begin{array}{lll}
\, [x_k,e]=-k(m+1-k)x_{k-1}, & k=1, \dots, m. &\\
\, [x_k,f]=x_{k+1},  & k=0, \dots, m-1, & \\
\, [x_k,h]=(m-2k)x_k & k=0, \dots, m,&\\
\, [e,h]=2e, & [h,f]=2f, &[e,f]=h, \\
\, [h,e]=-2e& [f,h]=-2f, & [f,e]=-h,\\
 \end{array}$$
It is easy to see that the algebra $L$ is a simple Leibniz algebra and the quotient algebra $L/I$ is isomorphic to the simple Lie algebra $sl_2.$ Let $d$ be a derivation of the Leibniz algebra $L$.
A straightforward calculation shows that $d=R_a+\alpha+\triangle,$ where $$a\in <e,f,h>, \ \alpha(<e, f , h>)=0, \ \alpha(x_k)=\lambda x_{k}, \ 0\leq k \leq m, \ \lambda\in \mathbb{\mathbb{C}}$$ and

for $m\neq 2: \ \ \triangle(x_k)=0, \ \triangle(e)=\triangle(f)=\triangle(h)=0,$

for $m=2: \ \ \triangle(x_k)=0, \ \triangle(e)=2\mu x_0, \ \triangle(f)=\mu x_2, \ \triangle(h)=\mu x_1, \ \mu \in \mathbb{\mathbb{C}}.$
\end{exam}

This example motivates to state and prove the following theorem.
\begin{thm} \label{thmMAIN} Let $L$ be a simple Leibniz algebra over a field
$F$ of zero characteristic. Then any derivation $d$ of $L$ can
be represented as $d=R_a+\alpha+\triangle,$ where $a\in G,$
$\alpha : I \rightarrow I, \ \triangle: G\rightarrow I$ and $ \alpha([x,y])=[\alpha(x),y]$ for all $x, y
\in L$. Moreover, the $\alpha$ is either zero or $\alpha(I)=I.$
\end{thm}
\begin{proof} Due to Theorem \ref{thmBarnes} we have $L=G\dot{+}I,$ i.e. $L$ is a
semidirect sum of simple Lie algebra $G$ and the ideal $I.$
For $L$ we put $L_0:=G, \ L_1:=I$ and $L_i =0$ for $i>1$ or $i < 0.$ Then we get $\mathbb{Z}$-graded algebra $L =\bigoplus_{i\in\mathbb{Z}}L_{i}.$
This gradation induces $\mathbb{Z}$-gradation on $Der(L)$ as follows (see \cite{Omirov}, \cite{KHAKIMDJANOV}):
$$Der(L)=\bigoplus_{i\in\mathbb{Z}}Der(L)_{i}$$
where $Der(L)_i=\{ d\in Der(L) \ | \ d : L_j \mapsto L_{i+j}\}.$

Clearly, $Der(L)_i =0$ for  $i \leq -2$ and $i \geq 2.$ Therefore, we get $$Der(L)= Der(L)_{-1} \oplus Der(L)_{0} \oplus Der(L)_{1}.$$

Thus, for a $d\in Der(L)$ we have $d=d_{-1}+d_0+d_1.$

Since $d_{-1}([x,x]) = [x, d_{-1}(x)] + [d_{-1}(x), x] \in I$ we conclude that $d_{-1}(I)\subseteq I.$ At the same time by the definition $d_{-1}(I) \subseteq G.$ Consequently, $d_{-1}=0$. Therefore, $$Der(L)=Der(L)_{0} \oplus Der(L)_{1}.$$

Let $x,y \in G.$ Then
$$d_0([x,y])+d_1([x,y])=d([x,y])=[d(x),y]+[x,d(y)]$$
$$=[d_0(x),y]+[x,d_0(y)]+[d_1(x),y]+[x,d_1(y)].$$ This gives
$$d_0([x,y])=[d_0(x),y]+[x,d_0(y)].$$

Let now $x\in G, \ y\in I.$ Then
$$d_0([x,y])=[d_0(x),y]=[x,d_0(y)]=0.$$ Therefore,
$$d_0([x,y])=[d_0(x),y]+[x,d_0(y)].$$

Let $x\in I, \ y\in G.$ Then we have
$$d_1([x,y])=[d_1(x),y]=[x,d_1(y)]=0,$$
$$d_0([x,y])=d([x,y])=[d(x),y]+[x,d(y)]=[d_0(x),y]+[x,d_0(y)].$$

Therefore, $d_0$ is a derivation of the algebra $L.$ Consequently, $d_1$ is also a derivation.

Let us consider $\widehat{d}=d_{0{|_G}}.$ Then $\widehat{d}$ is a derivation of the Lie algebra $G.$ Since for simple Lie algebras any derivation is inner there exists an element $a\in G$ such that
$\widehat{d}=R_{a}.$  Hence, we have $d_0(x)=\widehat{d}(x)=R_a(x)$ for $x\in G$.

Denoting by $\alpha = d_0-R_a$ we have
$$\alpha(x): = d_0(x)-R_a(x)=\left\{\begin{array}{cc}0 &  x\in G,\\
\in I & x \in I.\end{array}\right.$$

Note that $\alpha(I)\subseteq I \subseteq Ann_r(L)=\{x\in L| [y,x]=0\}$ and $\alpha$ is a derivation being a difference of two derivations. Therefore,
\begin{equation}\label{eq1}
\alpha([x,y])=[\alpha(x),y], \ x, y \in L.
\end{equation}
This implies that
$\alpha(I)$ is an ideal of $L$. Denoting $\triangle = d_1$ we complete the proof of theorem.
\end{proof}

The following proposition shows that when the ground field is the field of complex numbers field the $\alpha$ is a scalar map.

\begin{prop} Let $L$ be a complex simple Leibniz algebra. Then any derivation $d$ of $L$ can be represented as $d=R_a+\alpha+\triangle,$ where $a\in G,$ $\triangle: G\rightarrow I, \ \alpha = \lambda id_{|_ I}$ for some $\lambda \in \mathbb{C}.$
\end{prop}
\begin{proof} The decomposition $d=R_a+\alpha+\triangle$ follows from Theorem \ref{thmMAIN}. Due to the simplicity of $L$ the ideal $I$ is an irreducible right $G$-module. The property (1) shows that the $\alpha : I \rightarrow I$  is $G$-module homomorphism. Now applying Schur's lemma we derive that $\alpha = \lambda id_{|_I}$ for some $\lambda \in F.$
\end{proof}

 Let $\{e_1, \dots, e_n, x_1, \dots, x_{m}\}$ be a basis of $L$ such that $\{x_1, \dots, x_{m}\}$ is a basis of $I$ and $\{e_1, \dots, e_n\}$ is a basis of
$G$ which generates a simple Lie subalgebra.
The theorem below describes the map $\triangle.$

\begin{thm} Let $L$ be a complex simple Leibniz algebra. Then any derivation $d$ of $L$ can be represented as $d=R_a+\alpha+\triangle,$ where $a\in G,$ $\triangle: G\rightarrow I, \ \alpha = \lambda id_{|_I}$ for some $\lambda \in \mathbb{C}.$ In addition, if $n\neq m,$ then $\triangle=0.$ If $n=m,$ then either $\triangle(G)=I$ or $\triangle(G)=0.$
\end{thm}
\begin{proof}  Let
$$
[e_i,e_j]=\sum_{k=1}^{n}c_{ij}^ke_k, \quad 1 \leq i\neq j \leq n,$$
$$\triangle(e_i)=\sum_{s=1}^m\gamma_{is}x_s, \quad 1\leq i \leq n.$$
%

Since $\triangle(e_i)\in I \subseteq Ann_r(L)$
%
from
$$\triangle([e_i,e_j])=[\triangle(e_i),e_j]+[e_i,\triangle(e_j)]$$
%
%
%
we have
\begin{equation} \label{eq5}
\sum_{k=1}^{n}c_{ij}^k\triangle(e_k)=[\triangle(e_i),e_j].
\end{equation}

Let us consider the subspace $A$ of $I$ generated by the vectors $\{\triangle(e_1), \triangle(e_2), \dots, \triangle(e_n)\}:$ $A=span_F\{\triangle(e_1), \triangle(e_2), \dots, \triangle(e_n)\}.$
Thanks to (\ref{eq5}), the subspace $A$ is a right $G$-module. Since in the case of simple Leibniz algebras $I$ is irreducible $G$-module we conclude that either $A=I$ or $A=0$. Clearly, if $n\neq m$, then $A=\{0\}$, i.e., $\gamma_{ij}=0$ with $1\leq i, j \leq m,$ whereas if $n=m$ we have either $A=\{0\}$ or $A=I.$
\end{proof}

\subsection{Derivations of some semisimple Leibniz algebras.}

Let $L$ be a semisimple Leibniz algebra. Then Theorem \ref{thmBarnes} implies that $L$ can be written as a semidirect sum of semisimple Lie algebra $G$ and the ideal $I$ generated by squares: $L=G\dot{+}I.$ Unfortunately, the decomposition of a semisimple Leibniz algebra into direct sum of simple ideals is not true (as an example one can take the algebra from Theorem \ref{thm2}). However, that class of semi-simple Leibniz algebras which can be written as the direct sum of simple Leibniz algebras is of great interest. Therefore now we consider semisimple Leibniz algebras which are decomposed into the direct sum of simple Leibniz algebras: $L=\bigoplus_{i=1}^s L_i.$  One has the following
\begin{thm} Any derivation $d$ of semisimple Leibniz algebra $L=\bigoplus_{i=1}^s L_i$ given above is represented in a sum of derivations of $L_i,$ i.e. $d=\sum\limits_{i=1}^s d_i$, where each $d_i$ has the form in Theorem $\ref{thmMAIN}.$
\end{thm}
\begin{proof}
For $L_i$ we have $L_i=G_i\dot{+}I_i.$ Applying Theorem \ref{thmMAIN} as $G=\bigoplus_{i=1}^s G_i$ and $I=\bigoplus_{i=1}^s I_i$ we derive that any derivation $d$ of semisimple Leibniz algebra $L$ can be represented as follows:
$$d=R_a+\alpha+\triangle,$$ where $a\in G,$ $\alpha : I \rightarrow I, \ \triangle: G\rightarrow I,$ and $ \alpha([x,y])=[\alpha(x),y]$ for all $x, y\in L.$

Let $\alpha_{i,j}=\alpha|_{I_i} : I_i \rightarrow I_j.$ Then $\alpha=\sum_{i,j=1}^s\alpha_{i,j}.$
It is obvious that $\alpha_{i,j}=0$ for $i\neq j$ since
$$[\alpha_{i,j}(I_i),G_{j}]=[\alpha(I_i),G_{j}]=
\alpha([I_i,G_{j}])=0.$$ Hence, $\alpha=\sum\limits_{i=1}^s\alpha_{i,i}.$

Moreover, each of $I_j$ is an irreducible $G_j$-module since $\alpha_{j,j}([x,y])=[\alpha_{j,j}(x),y], \ x\in I_j, \ y\in G_j.$ Then applying Schur's lemma we get $\alpha_{j,j}=\lambda_j id{|_{I_j}}, \ 1\leq j \leq s$ for some $\lambda_j\in F.$

Now we deal with $\triangle=\sum_{i,j=1}^{s}\triangle_{i,j},$ where $\triangle_{i,j}=\triangle|_{G_i} : G_i \rightarrow I_j.$
Due to
$$[\triangle_{i,j}(G_i), G_j]=[\triangle(G_i), G_j]=[\triangle(G_i), G_j]+[G_i, \triangle(G_j)]=\triangle([G_i, G_j])=0$$
we have $\triangle_{i,j}=0$ for $i\neq j.$ Therefore, $\triangle=\sum\limits_{i=1}^s\triangle_{i,i}.$
\end{proof}

Let us now consider the semisimple Leibniz algebra $L$ from Theorem \ref{thm2}. Remind that as has been mentioned above this algebra is an example of semisimple Leibniz algebra which does not admit a decomposition into direct sum of simple ideals.

\begin{prop} Any derivation $d$ of the algebra $L$ is represented as $d=R_a+\alpha,$
where $a\in sl_2^1\oplus sl_2^2,$ $\alpha=\alpha_{1,1}+\alpha_{1,2}$ with $\alpha_{1,i} : I_{1,i} \rightarrow I_{1,i}$ and $\alpha_{1,1}=\lambda_1 id{|_{I_{1,1}}}, \ \alpha_{1,2}=-\lambda_1 id{|_{I_{1,2}}}, \  \lambda_1\in F.$
\end{prop}
\begin{proof}
Since $L$ is semisimple, we have already proved that for a derivation $d$ of $L$ one has the decomposition:
$$d=R_a+\alpha+\triangle,$$
where $a\in G,$ $\triangle: G\rightarrow I, \ \alpha : I \rightarrow I,$ and $ \alpha([x,y])=[\alpha(x),y]$ for all $x, y\in L$.

Introducing the notation $\triangle_{i,j}=\triangle|_{sl_2^i } : sl_2^i \rightarrow I_{1,j}$ we have $\triangle=\sum_{i,j=1}^{2}\triangle_{i,j}.$

Since $[\triangle_{2,1}(sl_2^2), sl_2^1]\subseteq I_{1,1}$ and $[\triangle_{2,2}(sl_2^2), sl_2^1]\subseteq I_{1,2}$ the equality below
$$[\triangle_{2,1}(sl_2^2), sl_2^1]+[\triangle_{2,2}(sl_2^2), sl_2^1]=[\triangle(sl_2^2), sl_2^1]$$
$$=[\triangle(sl_2^2), sl_2^1]+[sl_2^2, \triangle(sl_2^1)]=\triangle([sl_2^2, sl_2^1])=0.$$
shows that $\triangle_{2,1}=\triangle_{2,2}=0.$

Therefore, $\triangle=\triangle_{1,1}+\triangle_{1,2}.$

On the other hand, we have
$$[\triangle_{1,1}(sl_2^1), sl_2^2]+[\triangle_{1,2}(sl_2^1), sl_2^2]=[\triangle(sl_2^1), sl_2^2]+[sl_2^1, \triangle(sl_2^2)]=\triangle([sl_2^1, sl_2^2])=0.$$
That means
\begin{equation} \label{eq111}
[\triangle_{1,1}(x) +\triangle_{1,2}(x), y]=0, \quad x\in sl_2^1, \ y\in sl_2^2.
\end{equation}

Let us introduce the following notations:
$$\triangle_{1,1}(e_1)=\sum_{k=0}^{m}\beta_{1,k}^{e_1}x_k^1, \ \ \triangle_{1,1}(h_1)=\sum_{k=0}^{m}\beta_{1,k}^{h_1}x_k^1, \ \ \triangle_{1,1}(f_1)=\sum_{k=0}^{m}\beta_{1,k}^{f_1}x_k^1,$$

$$\triangle_{1,2}(e_1)=\sum_{k=0}^{m}\beta_{2,k}^{e_1}x_k^1, \ \ \triangle_{1,2}(h_1)=\sum_{k=0}^{m}\beta_{2,k}^{h_1}x_k^1, \ \ \triangle_{1,2}(f_1)=\sum_{k=0}^{m}\beta_{2,k}^{f_1}x_k^1.$$

Then thanks to $[x_k^1,e_2]=x_k^2$ along with (\ref{eq111}) we derive
$$\beta_{1,k}^{e_1}=-\beta_{2,k}^{e_1}, \quad \beta_{1,k}^{h_1}=-\beta_{2,k}^{h_1}, \quad \beta_{1,k}^{f_1}=-\beta_{2,k}^{f_1}, \quad
 1\leq k \leq m.$$
As a result we get
$\triangle_{1,1}=-\triangle_{1,2}\ \mbox{which gives}\ \ \triangle=0.$
Therefore,
$$d=R_a+\alpha.$$

Let
$\alpha_{i,j}=\alpha|_{I_{1,i}} : I_{1,i} \rightarrow I_{1,j}.$ Then $\alpha=\alpha_{1,1}+\alpha_{1,2}+\alpha_{2,1}+\alpha_{2,2}.$

Consider
$$[\alpha_{1,1}(x_k^1),e_2]+[\alpha_{1,2}(x_k^1),e_2]=[\alpha(x_k^1),e_2]=\alpha([x_k^1,e_2])=
\alpha(x_k^2)=\alpha_{2,1}(x_k^2)+\alpha_{2,2}(x_k^2).$$

Since $[x_k^2,e_2]=0,$ we have

$$[\alpha_{1,1}(x_k^1),e_2]=\alpha_{2,1}(x_k^2)+\alpha_{2,2}(x_k^2).$$
Then due to $[x_k^1,e_2]=x_k^2$ we get $[\alpha_{1,1}(x_k^1),e_2]\in I_{1,2}.$
Therefore,
$$[\alpha_{1,1}(x_k^1),e_2]=\alpha_{2,2}(x_k^2), \quad \alpha_{2,1}=0.$$

Similarly, due to
$$[\alpha_{2,2}(x_k^2),f_2]=[\alpha_{2,1}(x_k^2),f_2]+[\alpha_{2,2}(x_k^2),f_2]=[\alpha(x_k^2),f_2]$$
$$=\alpha([x_k^2,f_2])=
-\alpha(x_k^1)=-\alpha_{1,1}(x_k^1)-\alpha_{1,2}(x_k^1),$$
along with $[x_k^2,f_2]=-x_k^1$ we obtain $[\alpha_{2,2}(x_k^2),f_2]\in I_{1,1}.$

Hence, $\alpha_{1,2}=0$ and
\begin{equation}\label{eq1111}
[\alpha_{2,2}(x_k^2),f_2]=-\alpha_{1,1}(x_k^1).
\end{equation}

Thus, we conclude that $$\alpha=\alpha_{1,1}+\alpha_{2,2}.$$

Applying Schur's lemma for the irreducible $sl_2^1$-modules $I_{1,1}$ and $I_{2,2}$ we get
\begin{equation}\label{eq2222}
\alpha_{1,1}=\lambda_1 id|_{I_{1,1}}, \quad \alpha_{2,2}=\lambda_2 id|_{I_{2,2}}, \quad  \lambda_1, \lambda_2\in F.
\end{equation}
Now substituting (\ref{eq2222}) into (\ref{eq1111}) we derive $\lambda_2=-\lambda_1,$ which completes the proof.
\end{proof}

\begin{thm} \label{thm3.7} Let $L$ be a complex semisimple Leibniz algebra whose quotient algebra $L/I$ is isomorphic to a simple Lie algebra $G.$  Then any derivation $d$ of $L$ can be represented as $d=R_a+\alpha+\triangle,$ where $a\in G, \ \alpha : I \rightarrow I, \ \triangle: G\rightarrow I$ and $ \alpha([x,y])=[\alpha(x),y]$ for all $x, y \in L$. Moreover, the module $I$ is written as a direct sum of irreducible $G$-modules $I_i$ $i=1,2,...,s$ and $\alpha=\sum_{i,j=1}^{s}\alpha_{i,j},$ where $\alpha_{i,j} : I_i \rightarrow I_j, \ 1\leq i, j \leq s$ with $\alpha_{i,i}=\lambda_i id_{|I_i}, \ 1 \leq i \leq s$.
\end{thm}
\begin{proof} The decomposition $d=R_a+\alpha+\triangle$ is obtaining similar as in the proof of Theorem \ref{thmMAIN}. Since $G$ is simple Lie algebra then due to Theorem $\ref{reducible}$ we conclude that $G$-module $I$ is decomposed into direct sum of irreducible $G$-submodules.

We set $I=\bigoplus_{i=1}^{s} I_i,$ where $I_i$ is irreducible $G$-module and $\alpha_{i,j}=\alpha|_{I_i} : I_i \rightarrow I_j, \ 1\leq i, j \leq s.$

Then $\alpha=\sum_{i,j=1}^{s}\alpha_{i,j}$ and $(\ref{eq1})$ implies that
\begin{equation} \label{9}
\sum_{i=1}^{s}[\alpha_{i,j}(x),y]=\sum_{i=1}^{s}\alpha_{i,j}([x,y]), \ x\in I_i, y\in G \quad 1\leq j \leq s.
\end{equation}

If $j=i$ then $(\ref{9})$ gives $[\alpha_{i,i}(x),y]=\alpha_{i,i}([x,y])$ for any $x\in I_i, \ y\in G.$ Then applying Schur's lemma we get
$\alpha_{i,i}=\lambda_i id_{|I_i}, \ 1 \leq i \leq s.$ The other $\alpha_{i,j}, \ 1\leq i\neq j, \ 1\leq k \leq s$ are maps subjected to the conditions $(\ref{9}).$
\end{proof}

\section*{ Acknowledgements}
The idea to use Schur's lemma is due to K. Kudaybergenov to whom the authors are grateful. The third named author thanks to Ministerio de Econom\'ia y Competitividad (Spain) for a support via grant MTM2013-43687-P (European FEDER support included)

\end{document}